\documentclass{amsart}

\usepackage{mathrsfs}
\usepackage{stmaryrd}
\usepackage{amscd}
\usepackage{amssymb}
\usepackage[alphabetic]{amsrefs}
\usepackage[all]{xy}
\usepackage{mathtools}

\title[Vanishing theorems for some simple Shimura varieties]
{Vanishing theorems for the mod $p$ cohomology of some simple Shimura varieties}
\author{Teruhisa Koshikawa}
\address{Research Institute for Mathematical Sciences, Kyoto University}
\email{teruhisa@kurims.kyoto-u.ac.jp}

\theoremstyle{plain}
\newtheorem{thm}{Theorem}[section]
\newtheorem{lem}[thm]{Lemma}
\newtheorem{prop}[thm]{Proposition}

\theoremstyle{definition}
\newtheorem{rem}[thm]{Remark}

\newtheorem{exam}[thm]{Example}

\newcommand\fA{\mathfrak A}
\newcommand\fm{\mathfrak m}
\newcommand\fX{\mathfrak X}

\newcommand\bA{\mathbb A}
\newcommand\bC{\mathbb C}
\newcommand\bT{\mathbb T}

\newcommand\bF{\mathbf F}
\newcommand\bQ{\mathbf Q}
\newcommand\bZ{\mathbf Z}

\newcommand\cK{\mathcal K}
\newcommand\cO{\mathcal O}

\newcommand{\ra}{\rightarrow}
\newcommand{\ov}{\overline}
\newcommand{\wt}{\widetilde}

\DeclareMathOperator{\diag}{diag}
\DeclareMathOperator{\GL}{GL}		
\DeclareMathOperator{\SL}{SL}		
\DeclareMathOperator{\Frob}{Frob}
\DeclareMathOperator{\End}{End}
\DeclareMathOperator{\Gal}{Gal}

\DeclareMathOperator{\Spl}{Spl}
\DeclareMathOperator{\Sym}{Sym}
\DeclareMathOperator{\ur}{ur}

\newcommand{\et}{\mathrm{\acute{e}t}}	

\begin{document}

\begin{abstract}
We show that the mod $p$ cohomology of a simple Shimura variety treated in Harris-Taylor's book vanishes outside a certain nontrivial range after localizing at any non-Eisenstein ideal of the Hecke algebra. In cases of low dimensions, we show the vanishing outside the middle degree under a mild additional assumption.  
\end{abstract}

\maketitle

\section{Introduction}

Let $F$ be a CM field that contains an imaginary quadratic field $\cK$. 
Let $G$ be a similitude unitary group that is associated with a division algebra $B$ with the center $F$ of dimension $n^2$ and an involution of the second kind, so that it gives rise to Kottwitz's simple Shimura variety $X_K$ for a fixed sufficiently small level $K$ defined over the reflex field $E$. 

Fix a prime number $p$.
Let $\fm$ be a system of Hecke eigenvalues appearing in $H^i_{\et} (X_{K, \overline{E}}, \overline{\bF}_p)$ for some $i$. 
Caraiani and Scholze \cite{CS} constructed a semisimple Galois representation
\[
\rho_{\fm}\colon \Gamma_{F}\coloneqq\Gal(\ov{F}/F) \to \GL_n (\overline{\bF}_p)
\]
associated with $\fm$. (Our normalization of $\rho_{\fm}$ is ``geometric".) Their proof also provides a character $\chi_{\fm}\colon \Gamma_{\cK} \to \overline{\bF}_p^\times$ corresponding to the similitude factor; see the main text.

Assume that the signature of $G$ is $(0,n)$ except at one infinite place $\tau_0\colon F\to \bC$. 
Let $\ell\neq p$ be a prime such that everything is unramified at $\ell$ and $\ell$ splits over $\cK$. 
Let $v$ be a finite place of $F$ dividing $\ell\neq p$, and fix an embedding $\Gamma_{F_v}\to \Gamma_F$. In particular, the restriction of $\rho_{\fm}$ to $\Gamma_{F_v}$ is unramified. 
All lifts of geometric Frobenius $\Frob_v$ at $v$ have the same image under $\rho_{\fm}$, and we write $\rho_{\fm}(\Frob_v)$ for the image by abuse of notation. Moreover, the conjugacy class of $\rho_{\fm}(\Frob_v)$ is independent of the choice of $\Gamma_{F_v}\to \Gamma_F$. 
Let $\alpha_{v, 1}, \dots, \alpha_{v, n}$ be the set of generalized\footnote{This means we count usual eigenvalues with multiplicities being the dimensions of the corresponding generalized eigenspaces.}  eigenvalues of $\rho_{\fm}(\Frob_v)$. 
We say that $\rho_{\fm}$ is generic at $v$ if $\alpha_{v, i} \neq q_v \alpha_{v, j}$ for $i\neq j$, where $q_v$ is the size of the residue field of $v$. 

The main result of Caraiani-Scholze's work \cite{CS} in this setting is the following vanishing theorem:

\begin{thm}[\cite{CS}*{1.5, 6.3.3}]
If $\rho_{\fm}$ is generic at some $v$, then $i=\dim X_K$. Namely, $H^*_{\et}(X_{K, \ov{E}}, \ov{\bF}_p)_{\fm}$ vanishes outside the middle degree. 
\end{thm}

\begin{rem}
In \cite{CS}, $\rho_{\fm}$ is assumed to be decomposed generic at $v$; a slightly stronger condition than being generic. But their proof can be modified easily to cover generic ones; see also \cite{CS2}, especially the proof of Corollary 5.1.3. 
\end{rem}

Now, assume the signature at $\tau_0$ is $(1, n-1)$; this is essentially the Harris-Taylor case \cite{HT}. (There are some technical additional assumptions in \cite{HT}.) So, the reflex field equals $F$ and $\dim X_K=n-1$. 
In the Harris-Taylor case, the above vanishing theorem is also proved in \cite{Boyer}*{4.7} by a different argument. 
In fact, he proved the following stronger result. 
Note first that the Galois action on $H^i_{\et} (X_{K, \overline{F}}, \overline{\bF}_p)_{\fm}\otimes \chi_{\fm}$ is unramified at $v$.  

\begin{thm}[\cite{Boyer}]\label{Key}
If $\alpha$ is an eigenvalue of $\Frob_v$ acting on the cohomology $H^i_{\et} (X_{K, \overline{F}}, \overline{\bF}_p)_{\fm}\otimes \chi_{\fm}$, then the multiset
\[
\alpha, \quad q_v \alpha, \dots, \quad q_v^{(n-1)-i}\alpha
\]
is a subset of the multiset $\{ \alpha_{v,1}, \dots, \alpha_{v, n}\}$ of generalized eigenvalues of $\rho_{\fm}(\Frob_v)$.  
\end{thm}

\begin{rem}\label{finite extension}
There is an immediate variant using a finite extension of $F$: let $F'$ be a finite extension of $F$, and let $v'$ be a finite place of $F'$ above $v$. Then, the theorem holds with $\Frob_{v}$ and $q_v$ replaced by $\Frob_{v'}$ and $q_{v'}$. Indeed, it follows from the theorem since $\Frob_{v'}=\Frob_v^{[k_{v'}:k_v]}$ and $q_{v'}=q_v^{[k_{v'}:k_v]}$. 

Later, the theorem or the variant above is used in the following way, combined with the Chebotarev density: let $g$ be an element of $\rho_{\fm}(\Gamma_{F'})$. Then, there exist infinitely many finite places $v'$ (resp. $v$) of $F'$ (resp. $F$) to which the variant can be applied and the conjugacy class of $g$ equals that of $\rho_{\fm}(\Frob_{v'})$. Moreover, we can make $q_{v'}=p_{v'}=v'|_{\bQ}$ as only such $v'$ contribute to nonzero Dirichelet density (actually this makes $p_{v'}=p_v\coloneqq v|_{\bQ}$ split in $\cK$, and we can apply the results above). 

Finally, note that if $F'$ contains $F(\zeta_p)$, then $q_{v'}\equiv 1 \mod p$ for $v'$ prime to $p$ and the statement of the variant simplifies. This simplification is very useful and will be used frequently.  
\end{rem}

\begin{rem}
Theorem \ref{Key} is actually not clearly stated in \cite{Boyer} but follows from an argument along the line of \cite{Boyer}*{4.14} by considering the greatest integer $i'\geq 0$ such that $H^{(n-1)-i'}_{\et}(X_{K, \ov{F}}, \ov{\bF}_p)_{\fm} \neq 0$ or (in fact, and, a posteriori) $H^{(n-1)-i'}_{\et}(X_{K, \ov{F}}, \ov{\bF}_p)_{\fm^{\vee}} \neq 0$, where $\fm^{\vee}$ is the ``dual" of $\fm$. 
It can be proved using the method of \cite{CS} as well. This will be discussed in a forthcoming article of the author. 
\end{rem}

\begin{rem}
Assume that $\ell$ splits completely in $F$. 
Then, any eigenvalue of Frobenius at $v$ acting on $H^i_{\et} (X_K, \overline{\bF}_p)_{\fm}\otimes \chi_{\fm}$ is a Frobenius eigenvalue of $\rho_{\fm}|_{\Gamma_{F_v}}$ by Wedhorn's congruence relation \cite{Wedhorn} and our normalization of $\rho_{\fm}$.

As a part of a mod $p$ analogue of the Arthur-Kottwitz conjecture, one would consider hypothetical Lefschetz operators $H^{j}_{\et}(X_{K, \ov{F}}, \overline{\bF}_p)_{\fm} \ra H^{j+2}_{\et}(X_{K, \ov{F}}, \overline{\bF}_p)_{\fm}(1)$ inducing isomorphisms $H^{(n-1)-i}_{\et}(X_{K, \ov{F}}, \overline{\bF}_p)_{\fm} \cong H^{(n-1)+i}_{\et}(X_{K, \ov{F}}, \overline{\bF}_p)_{\fm}(i)$. 
This would imply that each $\alpha, \dots, \ell^i\alpha$ is a Frobenius eigenvalue of $\rho_{\fm}$. Theorem \ref{Key} is stronger actually, and gives information of multiplicities; this may be also regarded as a part of the mod $p$ analogue of the Arthur-Kottwitz conjecture. 
\end{rem}

The main result of this note, which is deduced from Boyer's result, is the following:

\begin{thm}\label{general}
Let $X_K$ be Harris-Taylor's Shimura variety of dimension $n-1$~\cite{HT}. Let $\fm$ be a maximal ideal of the Hecke algebra contributing to the cohomology of $X_K$, and $\rho_{\fm}\colon \Gamma_F \ra \GL_n(\ov{\bF}_p)$ the associated Galois representation. If $\rho_{\fm}$ is irreducible, then $H^j_{\et}(X_{K, \ov{F}}, \overline{\bF}_p)_{\fm}$ vanishes for $j < n/2$ and $j > 2(n-1)- n/2$.      
\end{thm}

In particular, the cohomology localized at $\fm$ vanishes outside the middle degree if $\rho_{\fm}$ is irreducible and $n \leq 3$. While the case $n=4$ is difficult as $n$ is no longer prime, we can push the argument further if $n=5$:

\begin{thm}\label{small}
Suppose $n = 5$ and $\rho_{\fm}$ is irreducible. Then, $H^*_{\et}(X_{K, \ov{F}}, \ov{\bF}_p)_{\fm}$ vanishes outside the middle degree, except possibly when $p=5$ and every minimal noncentral normal subgroup of $\rho_{\fm}(\Gamma_{F(\zeta_p)})$ is an elementary abelian $r$-group with $r=2$ or $3$.  
\end{thm}

\begin{rem}
For instance, the above excludes a case where $p=5$ and $\rho_{\fm}(\Gamma_{F(\zeta_p)})$ is the semidirect product of $\{\textnormal{diag}(\pm 1, \dots, \pm 1)\}$ (of order $32$) and a permutation of order $5$. Thanks for the referee for pointing out an erroneous claim in the first version of the manuscript.  
\end{rem}

There are previous works in this direction including the works of Shin \cite{Shin:sc} (see also \cite{Fujiwara}), Emerton-Gee \cite{EG}, and Boyer \cite{Boyer}. The novelty here is that we only assume irreducibility of $\rho_{\fm}$. 
(Let us also mention that Boyer is claiming a rather strong vanishing result recently.)
For the proofs, we use Theorem \ref{Key} and also group-theoretic results from \cite{GM}, which, in full generality, rely on the classification of finite simple groups. 

It is easy to control $\rho_{\fm}$ with large image. Let us record the following remark. 
The argument passing to $F(\zeta_p)$ is very important throughout this note, and will be frequently used later as well. 

\begin{thm}\label{regss}
If the image $\rho_{\fm}(\Gamma_{F(\zeta_p)})$ of $\Gamma_{F(\zeta_p)}$ contains a regular semisimple element of $\GL_n(\ov{\bF}_p)$, then $H^*_{\et}(X_{K, \ov{F}}, \overline{\bF}_p)_{\fm}$ vanishes outside the middle degree. 
\end{thm}

\begin{proof}
Let $g\in \rho_{\fm}(\Gamma_{F(\zeta_p)})$ denote a regular semisimple element of $\GL_n(\ov{\bF}_p)$. 
We shall apply Theorem \ref{Key}, Remark \ref{finite extension} to $g$. So, we choose a prime-to-$p$ finite place $v'$ of $F'=F(\zeta_p)$ as in Remark \ref{finite extension}. In particular, $g$ belongs to the conjugacy class of $\rho_{\fm}(\Frob_{v'})$. 
Hence, $\rho_{\fm}(\Frob_{v'})$ is also regular semisimple. 
Observe that the multiset of eigenvalues of $\rho_{\fm}(\Frob_{v'})$ does not contain a subset of the form of $\{ \alpha, q_{v'}\alpha\}$ as $q_{v'}\equiv 1\mod p$. Therefore, $H^*_{\et}(X_{K, \ov{F}}, \overline{\bF}_p)_{\fm}$ vanishes outside the middle degree by Remark \ref{finite extension}.   
\end{proof}

\begin{exam}
Suppose $n$ is an odd prime, $p>2n-3$, and the restriction of $\rho_{\fm}$ to $\Gamma_{F(\zeta_p)}$ is irreducible, i.e., $\rho_{\fm}$ is irreducible and not induced from a character. Then, \cite{GM}*{1.6} says that the image of $\Gamma_{F(\zeta_p)}$ contains a regular semisimple element. So, $H^*_{\et}(X_{K, \ov{F}}, \overline{\bF}_p)_{\fm}$ vanishes outside the middle degree by Theorem \ref{regss}. 
\end{exam}

\begin{exam}\label{Boyer}
Another example satisfying the assumption of Theorem \ref{regss} is the case where the image of $\Gamma_F$ contains $\SL_n(\bF_p)$. 
(Note that $\SL_n(\bF_p)$ contains a regular semisimple element.)
Indeed, if $(n, p)\neq (2, 2), (2, 3)$, $\SL_n(\bF_p)$ is perfect and contained in the image of $\Gamma_{F(\zeta_p)}$. If $p=2$, then $F=F(\zeta_p)$ and there is nothing to prove. If $p=3$, $[F(\zeta_p):F]$ divides $2$ and $\SL_2(\bF_3) $ does not have a subgroup of index $2$, so it is contained in the image of $\Gamma_{F(\zeta_p)}$. 
\end{exam}

We also remark that Theorem \ref{Key} implies the following: 

\begin{prop}\label{induced from character}
Suppose
\begin{itemize}
    \item $\rho_{\fm}$ is irreducible and induced from a character of $\Gamma_E$ for some cyclic extension $F\subset E$ of degree $n$ contained in $F(\zeta_p)$, and 
    \item $[F(\zeta_p)\colon F]> n$.
\end{itemize} 
Then, $H^*_{\et}(X_{K, \ov{F}}, \overline{\bF}_p)_{\fm}$ vanishes outside the middle degree. 
\end{prop}

\begin{proof}
Suppose $H^i_{\et}(X_{K, \ov{F}}, \ov{\bF}_p)_{\fm}$ is nonzero for some $i<n-1$. 
Pick a generator $h\in \Gal (F(\zeta_p) / F)$. It maps to a generator of the quotient $\Gal (E/F)$ as well. 
We can write $h$ as the (geometric) Frobenius of some finite place $v$ of $F$ such that $q_v=p_v$. 
(In particular, $p_v$ splits in $\cK$.)
Let $\sigma\in \Gamma_{F_v}$ denote a lift of $\Frob_v$; so $\sigma$ maps to $h\in\Gal (F(\zeta_p)/ F)$. 
As $h$ maps to a generator of $\Gal (E/F)$, the characteristic polynomial of $\rho_{\fm}(\sigma)$ has the form of $X^n -c$. 
Combined with Theorem \ref{Key}, this implies that $q_v^n\equiv 1 \mod p$. 
However, as $h$ is a generator of $\Gal (F(\zeta_p)/F)\subset (\bZ/p\bZ)^\times$ and $[F(\zeta_p)\colon F]> n$, we have $q_v^n\equiv p_v^n\not\equiv 1 \mod p$.  

The dual argument settles the case $i>n-1$. 
\end{proof}

\noindent\textbf{Acknowledgements}.
I thank Tetsushi Ito for helpful conversations. 
I am grateful for the anonymous referee for pointing out several errors and making suggestions that greatly improved the text.
The main result was first presented at ``Arithmetic Geometry in Carthage" held in June 2019, where I had conversations with Ana Caraiani and Pascal Boyer on this topic.  I am grateful for the organizers for making such an opportunity. 
This work was supported by JSPS KAKENHI Grant Number 20K14284.
\section{Preliminaries}

\subsection{Setting}
Let $F=F^+ \cK$ be a CM field with a totally real field $F^+$ and an imaginary quadratic field $\cK$.  We fix an embedding $\cK \hookrightarrow \bC$. We consider a PEL datum $(B, *, V, (\cdot, \cdot))$ of type A such that   
\begin{itemize}
\item $B$ is a division algebra with center $F$ and $V\cong B$, and
\item the associated group $G$ has signature $(1, n-1)$ at one infinite place, and $(0,n)$ at the other infinite places, where $n^2=\dim_F B$. (The signature is calculated using the fixed embedding $\cK \hookrightarrow \bC$.) 
\end{itemize}

Fix a sufficiently small open compact subgroup $K=\prod_\ell K_\ell $ of $G(\bA_f)$. If $\ell$ splits in $\cK$, by choosing a place $y$ of $\cK$ over $\ell$, we have an isomorphism $G(\bQ_\ell) \cong \bQ_\ell^\times \times \prod_x B_x^{\textnormal{op}\times}$, where $x$ runs over the places of $F$ lying over $y$.  

Let $\Spl^{\ur}$ denote the set of unramified places $v$ of $F$ satisfying
\begin{itemize}
    \item $v$ does not divide $p$,
    \item $p_v=v|_\bQ$ splits in $\cK$, 
    \item $B$ splits at all places above $p_v$, and 
    \item $K_{p_v}$, as a subgroup of $\bQ_{p_v}^\times \times \prod_x B_x^{\textnormal{op}\times}$, decomposed into a product of $\bZ_{p_v}^\times$ and maximal open compact subgroups $K_x$ of $B_x^{\textnormal{op}\times}$. 
\end{itemize}
Let $\bT$ denote the Hecke algebra
\[
\bigotimes_{p_v \in \Spl^{\ur}|_{\bQ}} \bZ[G (\bQ_{p_v}) // K_{p_v}]. 
\]
If we identify $K_v$ with $\GL_n(\cO_{F_v})$, its factor at $v$ is generated by 
\[
T_{v, j} =K_v \diag (\underbrace{p_v^{-1}, \ldots, p_v^{-1}}_j, \underbrace{1, \ldots, 1}_{n-j})K_v.
\]
We write $c_{v}$ for the element of $\bT$ determined by $p_v^{-1} \in \bQ_{p_v}^\times$. 
Our choice of the Hecke operators is different from the one of \cite{Wedhorn}, \cite{EG}, \cite{Boyer}, \cite{CS}. 

We denote by $X_K$ the canonical model of the Shimura variety attached to $(B, *, V, (\cdot, \cdot))$ of level $K$, which is a smooth projective variety over $F$. 
(We use the convention that (a disjoint union of copies of) the canonical model admits a usual moduli interpretation.)
The mod $p$ cohomology of $X_{K, \ov{F}}$ is naturally a module of $\bT \times G_F$. 

\begin{thm}[\cite{CS}*{6.3.1}]
Let $\fm$ be a maximal ideal of $\bT$ such that, for some $i$, $H^i_{\et}(X_{K, \ov{F}}, \overline{\bF}_p)_{\fm}\neq 0$. Then, there is a (unique) semisimple Galois representation $\rho_{\fm}\colon \Gamma_F \ra \GL_n(\ov{\bF}_p)$ and a character $\chi_{\fm}\colon \Gamma_{\cK} \ra \ov{\bF}_p^\times$, both unramified at $v\in \Spl^{\ur}$, such that the characteristic polynomial of $\rho_{\fm}(\Frob_v)$ for $v$ is given by
\[
\sum^n_{j=0}(-1)^j q_v^{j(j-1)/2}\ov{T}_{v,j} X^{n-j}
\]
and  $\chi_{\fm} (\Frob_{p_v}) =\ov{c}_v^{-1}$. 
where $\ov{T}_{v,j}$ and $\ov{c}_v$, denote the image of $T_{v,j}$ and $c_v$ in $\bT /\fm \cong \ov{\bF}_p$ respectively. 
\end{thm}

\begin{proof}
The existence of $\rho_{\fm}$ is proved in \cite{CS} or \cite{Boyer} up to normalization; our $\rho_{\fm}$ is a twist of the dual of the representation they constructed. The existence of $\chi_{\fm}$ can be proved by the same method. Namely, we find a characteristic $0$ lift $\Pi$ of $\fm$ at first; $\Pi$ is a $C$-algebraic cuspidal automorphic representation of $G$, and its stable base change is a $C$-algebraic isobaric automorphic representation of $\cK^\times \times \GL_n (F)$ of the form of $\psi\otimes \Pi^1$. The first factor $\psi$ gives rise to a character $\wt{\chi}_{\fm}^{-1}\colon \Gamma_{\cK} \ra \ov{\bQ}_p^{\times}$ via the class field theory. (The Artin map is normalized so that uniformizers correspond to lifts of geometric Frobenius.) The reciprocal of the reduction mod $p$ of $\wt{\chi}_{\fm}^{-1}$ is $\chi_{\fm}$. 
\end{proof}

Throughout this note, we regard $\chi_{\fm}$ as a character of $\Gamma_F (\subset \Gamma_{\cK})$ as well.  

\subsection{The congruence relation}
The congruence relation is not logically needed (in the sense that Theorem \ref{Key} is stronger) but we give a short explanation to clarify our notation and conventions. 
For every $v\in\Spl^{\ur}$, there is a canonical integral model $\fX_K$ of $X_K$, which is smooth and projective over $\cO_{F_v}$. The action of $\bT\times \Gamma_{F_v}$ passes to the mod $p$ cohomology of the special fiber $X_{K, \ov{k(v)}}$ of the canonical integral model. In particular, the Galois action on $H^i_{\et}(X_{K, \ov{F}}, \overline{\bF}_p)$ is unramified at $v$. 

Assume that $p_v$ splits completely in $F$. 
If we look at the Frobenius action on $H^i_{\et}(X_{K, \ov{F}}, \overline{\bF}_p)_{\fm}\otimes \chi_{\fm}$, the main result of \cite{Wedhorn} implies the following relation:
\[
\sum^n_{j=0}(-1)^j q_v^{j(j-1)/2}T_{v,j} \Frob_v^{n-j}=0. 
\]
The formula is stated incorrectly (or imprecisely) in \cite{EG}*{3.3.1} and \cite{Boyer}*{4.2}:
\begin{itemize}
\item The Hecke correspondence in \cite{Wedhorn} is a left action (as a correspondence), while the Hecke action on the Shimura variety is a right action. This is why we change the choice of the Hecke operator. 
\item The twist by $\chi_{\fm}$ is needed. 
\end{itemize}

\section{Proof of Theorem \ref{general}}
Suppose $H^i_{\et}(X_{K, \ov{F}}, \overline{\bF}_p)_{\fm}\neq 0$ for some $i < n/2$, and let $\rho$ be an irreducible constituent of $H^i_{\et}(X_{K, \ov{F}}, \overline{\bF}_p)_{\fm}\otimes \chi_{\fm}$ as a representation of $\Gamma_F$.  

Suppose that $\rho$ is a character $\chi$. Then, by Theorem \ref{Key} and Remark \ref{finite extension}, $\chi(g)\in \overline{\bF}_p^\times$ for any $g\in \Gamma_{F(\zeta_p)}$ appears in the set of generalized eigenvalues of $\rho_{\fm}(g)$ with multiplicity $\geq n-i$. Therefore, $(\rho_{\fm}\otimes \chi^{-1})(g)$ has a generalized eigenvalue $1$ with multiplicity $\geq n-i$. Since $n-i > n/2$, it contradicts to the following theorem. (This discussion also works for $i=n/2$.)

\begin{thm}[\cite{GM}*{1.5.(a)}]\label{GM:normal}
Let $H\subset \GL_n(\ov{\bF}_p)$ be a finite group whose action on $\ov{\bF}_p^n$ is irreducible. For any nontrivial normal subgroup $H'$ of $H$, there exists semisimple $h\in H'$ such that the multiplicity of $1$ in the set of eigenvalues of $h$ is less than $n/2$. 
\end{thm}

\begin{proof}
In \cite{GM}, this is stated with $H'=H$. The proof actually finds $h$ in any given minimal normal subgroup $N$ of $H$. 
\end{proof}

Suppose $\dim \rho \geq 2$. Then, we claim that $\rho|_{\Gamma_{F(\zeta_p)}}$ is not unipotent modulo scalar, namely $\rho(h)$ is not unipotent modulo scalar for some $h\in \Gamma_{F(\zeta_p)}$. 
Indeed, assume that $\rho(h)$ is unipotent modulo scalar for every $h\in \Gamma_{F(\zeta_p)}$. 
Set $\ov{H}\coloneqq\rho (\Gamma_{F(\zeta_p)})/ (\textnormal{scalar})$; this is a $p$-group because the order of any element is a power of $p$. If $\ov{H}$ is nontrivial, $\ov{H}$ has a nontrivial center $Z$. If $\widetilde{Z}$ denotes the inverse image of $Z$ in $\rho (\Gamma_{F(\zeta_p)})$, then $\widetilde{Z}$ is abelian. (Consider the Jordan decomposition of elements of $\widetilde{Z}$.) Moreover, the restriction of $\rho$ to $\widetilde{Z}$ is semisimple since $\widetilde{Z}$ is normal in $\rho (\Gamma_{F(\zeta_p)})$ and $\rho (\Gamma_{F(\zeta_p)})$ is normal in $\rho (\Gamma_{F})$. This is impossible as $\widetilde{Z}$ contains an element of order $p$. Thus $\ov{H}$ is trivial, i.e., $\rho$ is scalar on $\Gamma_{F(\zeta_p)}$. 
Then, $\rho(\Gamma_F)$ is abelian as $F(\zeta_p)$ is a cyclic extension over $F$. Contradiction. 

So, there exists $h\in \Gamma_{F(\zeta_p)}$ such that $\rho(h)$ has at least two \emph{distinct} eigenvalues, say $\alpha, \beta$. Each has multiplicity $\geq n-i$ in the multiset of generalized eigenvalues of $\rho_{\fm}(h)$. Thus, $\dim \rho_{\fm} \geq 2(n-i)>n$. Contradiction. 

If $i> 2(n-1)-n/2$, the vanishing of $H^i_{\et}(X_{K, \ov{F}}, \overline{\bF}_p)_{\fm}$ follows from the vanishing of $H^{2(n-1)-i}_{\et}(X_{K, \ov{F}}, \overline{\bF}_p)_{\fm^{\vee}}$ and the Poincar\'{e} duality because $\rho_{\fm^{\vee}}$ is also irreducible. 

\section{Proof of Theorem \ref{small}}

We may only consider cohomology below the middle degree because the duality preserves the condition that $\rho_{\fm}$ is irreducible (and the exceptional case stated in Theorem \ref{small}). 

Suppose that $H^{i}_{\et}(X_{K, \ov{F}}, \overline{\bF}_p)_{\fm}\neq 0$ for some $i < 4$. We will consider two cases: 
\begin{enumerate}
\item The restriction of $\rho_{\fm}$ to $\Gamma_{F(\zeta_p)}$ is irreducible. 
\item The restriction of $\rho_{\fm}$ to $\Gamma_{F(\zeta_p)}$ is not irreducible. 
\end{enumerate}

\subsection{Group-theoretic results}
We will use another group-theoretic result from \cite{GM}:

\begin{thm}[\cite{GM}*{1.7}]\label{GM}
Let $H$ be a finite nonabelian simple group and $p$ be a prime number. 
Assume that $(H, p)\neq (\fA_5, 5)$. Then, there exist $p'$-elements $x, y,z \in H$ with $xyz=1$ such that $H=\langle x, y\rangle$. 
\end{thm}

This will be combined with Scott's lemma:

\begin{thm}[\cite{Scott}]\label{Scott}
Let $H$ be a finite group acting on a finite-dimensional vector space $V$ over a field $k$. Assume that $x, y,z$ generate $H$ and satisfy $xyz=1$. Then, 
\[
\dim V + \dim V^{H} + \dim (V^{\vee})^H \geq \dim V^x + \dim V^y + \dim V^z,
\]
where $V^*$ denotes the space of fixed vectors under the action of $*$. 
\end{thm}

\subsection{Preliminary analysis}\label{pre}
Before dealing with the case (1), let us first analyze a slightly more general situation where $n$ is a prime and $\rho_{\fm}|_{\Gamma_{F(\zeta_p)}}$ is irreducible. 
The discussion below partly follows referee's suggestion.

Let $N$ be a minimal noncentral normal subgroup of $\rho_{\fm}(\Gamma_{F(\zeta_p)})$. 
We make the following hypothesis throughout \S \ref{pre}:
\begin{center}
\emph{$N$ is not a quasi-simple group.} 
\end{center}
We shall show that $\rho_{\fm}(\Gamma_{F(\zeta_p)})$ contains a \emph{regular semisimple} element in the following cases:
\begin{enumerate}
    \item[(i)] the restriction of $\rho_{\fm}$ to $N$ is irreducible.
    \item[(ii)] the restriction of $\rho_{\fm}$ to $N$ is not irreducible and $n\neq p$.
    \item[(iii)] the restriction of $\rho_{\fm}$ to $N$ is not irreducible, $n=p=5$, and $N$ is not an elementary abelian $r$-group for $r= 2,3$.
\end{enumerate}
Hence, Theorem \ref{regss} gives the vanishing outside the middle degree in these cases. 

Let us first observe that the hypothesis above implies that $N$ is not perfect in all cases: if $N$ is perfect, then $N$ is nonabelian and $\rho_{\fm}|_N$ is a faithful irreducible representation of $N$ of \emph{prime} dimension. However, this cannot happen because $N$ modulo the center $Z(\rho_{\fm}(\Gamma_{F(\zeta_p)}))$ of $\rho_{\fm}(\Gamma_{F(\zeta_p)})$ is the direct product of (more than one) isomorphic simple groups as $N/Z(\rho_{\fm}(\Gamma_{F(\zeta_p)}))$ is a minimal normal subgroup of $\rho_{\fm}(\Gamma_{F(\zeta_p)})/Z(\rho_{\fm}(\Gamma_{F(\zeta_p)}))$. 

Therefore, $[N,N]$ is central in $\rho_{\fm}(\Gamma_{F(\zeta_p)})$ by the minimality of $N$. Hence, $N$ is \emph{nilpotent}. 
Using the minimality again, we see that $N$ is an $r$-group for some prime $r$. 
Let us now study each case.

\begin{enumerate}
    \item[(i)] Assume $\rho_{\fm}|_N$ is irreducible. In particular, $N$ is a nonabelian nilpotent $r$-group. As $\rho_{\fm}|_N$ is irreducible and $n$ is prime, we see that $n=r$ and $p\neq n$. In fact, $\rho_{\fm}|_N$ is induced from a character. So, $N$ contains a regular semisimple element by \cite{GM}*{5.2}\footnote{$n$ is assumed to be odd in the reference, but it is not used in the argument of \cite{GM}*{5.2}.}.
    \item[(ii)] Assume $\rho_{\fm}|_N$ is not irreducible and $n\neq p$. Then, $\rho_{\fm}|_N$ is the direct sum of \emph{distinct} characters as $N$ is a \emph{noncentral} normal subgroup of $\rho_{\fm}(\Gamma_{F(\zeta_p)})$. So, $\rho_{\fm}|_{\Gamma_{F(\zeta_p)}}$ is induced from a character of a subgroup of $\rho_{\fm}(\Gamma_{F(\zeta_p)})$ containing $N$, and $\rho_{\fm}(\Gamma_{F(\zeta_p)})$ contains a regular semisimple element by \cite{GM}*{5.2}. 
    \item[(iii)] Assume $\rho_{\fm}|_N$ is not irreducible, $n=p=5$, and $N$ is not an elementary abelian $r$-group for $r= 2,3$. 
    Note that $N$ is a subgroup of the diagonal $(\overline{\bF}_5^\times)^5$ stable under a permutation $\tau$ of order $5$, so $N$ is an \emph{elementary abelian} $r$-group for the prime $r\ge 7$ by the minimality of $N$.
    (It suffices to observe that $N$ contains a non-scalar element of order $r$: take a non-scalar element $x\in N$ such that $x^r$ is a scalar. Then, $\tau (x) x^{-1}$ is a nontrivial element of order $r$ and it is not a scalar since $\tau$ has order $n$ and $n=p\neq r$.)
    We claim that there is a regular semisimple element inside $N$.
    
    Let us prove the claim. From now on, we identify $N$ with a nontrivial subrepresentation of $\bZ/5\bZ$ acting on $\bF_{r}^5$. If $r\not\equiv 4 \mod 5$, then the complement of the trivial representation in $\bF_r^5$ is either irreducible or the direct sum of 4 distinct characters, and it is easy to find a regular semisimple element as $r\geq 7$. 
    
    If $r\equiv 4 \mod 5$, then the complement of the trivial representation is the direct sum of two irreducible two-dimensional subrepresentations. An element $x$ of each subrepresentation can be written as the component-wise trace of $(a, a\zeta_5, a\zeta_5^2, a\zeta_5^3, a\zeta_5^4)$ for $a\in \bF_{r^2}$ and a choice of $\zeta_5$. It is easy to see that $x$ corresponds to a regular semisimple element if and only if all coordinates are distinct if and only if $\overline{a}/a \notin \{1, \zeta_5, \zeta_5^2, \zeta_5^3, \zeta_5^4\}$, where $\overline{a}$ denotes the conjugate of $a$. 
    Any norm 1 element in $\bF_{r^2}$ has the form of $\overline{a}/a$. Since the number of norm 1 elements in $\bF_{r^2}$ is $r+1$ and $r+1 >5$ by the assumption, we can find a regular semisimple $x$. 
\end{enumerate}

\subsection{The case (1)}
Now $n=5$ and assume $\rho_{\fm}|_{\Gamma_{F(\zeta_p)}}$ is irreducible as in (1). 
By the discussion on (i)-(iii) above with Theorem \ref{regss}, we only need to consider the case where \emph{some} minimal noncentral normal subgroup $N$ of  $\rho_{\fm}(\Gamma_{F(\zeta_p)})$ is \emph{quasi-simple}. Note that $N$ acts irreducibly on $\rho_{\fm}$. 
Let $\rho_0$ be an irreducible constituent of $H^{i}_{\et}(X_{K, \ov{F}}, \overline{\bF}_p)_{\fm}\otimes \chi_{\fm}$ as a representation of $\Gamma_F$. We regard $\rho_0$ as a representation of $N$; this is possible by \cite{EG}*{4.1.3}. (The action of $\Gamma_F$ on $\rho_0$ factors through $\rho_{\fm}(\Gamma_{F})$.) 

Now let $\rho$ be an irreducible constituent of $\rho_0$ as a representation of $N$. 
By Theorem \ref{Key} and Schur's lemma, the center $Z$ of $N$ acts on $\rho$ and $\rho_{\fm}$ by the same character. Therefore, $\rho_{\fm}\otimes \rho^{\vee}$ becomes a representation of $N/Z$, which is a \emph{simple} nonabelian group. 

Assume $N/Z\neq \fA_5$ or $p \neq 5$. Suppose first that $\rho$ is not isomorphic to $\rho_{\fm}|_N$. Then, we can apply Theorem \ref{GM} and Scott's lemma (Theorem \ref{Scott}), and there is an element $n_0$ of $N$ satisfying the following conditions:
\begin{itemize}
    \item $n_0$ is a $p'$-element. In particular, the action of $n_0$ on $\rho_{\fm}\otimes \rho^{\vee}$ is \emph{semisimple}. 
    \item $\dim (\rho_{\fm}\otimes \rho^{\vee})^{n_0}\leq (5\dim \rho) /3$. 
\end{itemize}
(Note that $Z$ is a prime-to-$p$ group since $Z \subset N \subset \rho_{\fm}(\Gamma_F)$ consists of scalars.) 
But, this contradicts to
\begin{lem}
$\dim (\rho_{\fm}\otimes \rho^{\vee})^{n_0} \ge 2 \dim \rho$. 
\end{lem}

\begin{proof}
Let $\lambda_1, \dots, \lambda_{\dim \rho}$ denote the eigenvalues of $\rho(n_0)$ and write $v_{\lambda, i}$ for an eigenvector corresponding to $\lambda_i$ so that it forms a basis of $\rho$. The dual basis is denoted by $v_{\lambda,i}^{\vee}$.  
Theorem \ref{Key} and Remark \ref{finite extension} imply that, for each $i$, $\rho_{\fm}(n_0)$ has eigenvalue $\lambda_i$ with multiplicity $\geq 2$ since $n_0 \in N \subset \rho_{\fm}(\Gamma_{F(\zeta_p)})$. 
Take two linearly independent eigenvectors $w_{i,1}, w_{i,2}$ of $\rho_{\fm}(n_0)$ with eigenvalue $\lambda_i$.
The following $2\dim \rho$ vectors
\[
w_{1,1}\otimes v_{\lambda,1}^{\vee},  w_{1,2}\otimes v_{\lambda,1}^{\vee}, \dots, 
w_{\dim \rho,1}\otimes v_{\lambda, \dim\rho}^{\vee},  w_{\dim \rho,2}\otimes v_{\lambda, \dim\rho}^{\vee}
\]
are linearly independent and fixed by $n_0$. 
\end{proof}

Next, suppose $\rho$ is isomorphic to $\rho_{\fm}|_N$. Then, $\rho_{\fm}\otimes \rho^{\vee}\cong \End (\rho_{\fm})$, as a representation of $N$, is self-dual and has $1$-dimensional subrepresentation and quotient representation given by the scalars and the trace map respectively, and there is no other trivial subrepresentation or quotient representation. So, again by Theorem \ref{GM} and Scott's lemma (Theorem \ref{Scott}), we get an inequality $(25+2) /3 \geq 2\dim \rho=10$, which is impossible. 

The only remaining case is $N/Z=\fA_5$ and $p=5$. Note that the only such quasi-simple group is $\fA_5$ itself or $\SL_2(\bF_5)$, which is a double covering of $\fA_5$. 

\begin{itemize}
\item Suppose $N=\fA_5$. 
There are only three isomorphism classes of irreducible representations in characteristic $5$, and one of them has dimension $5$; it must be $\rho_{\fm}$. The other two are the trivial representation and a faithful 3-dimensional representation defined over $\bF_5$. Whatever $\rho$ is, any element $g$ of order $3$ has an eigenvalue $1$. However, $\rho_{\fm}(g)$ has the eigenvalues $\{1, \zeta_3, \zeta_3, \zeta_3^2, \zeta_3^2\}$ and $1$ has the multiplicity one. This contradicts to Theorem \ref{Key} and Remark \ref{finite extension}. 
\item Suppose $N=\SL_2(\bF_5)$. Any irreducible representations in characteristic $5$ is given by the symmetric power $\Sym^k \bF_5^2$ of the standard representation of dimension $2$ for an integer $k\in [0,4]$. 
So, $\rho_{\fm}|_N$ must be isomorphic to $\Sym^4 \bF_5^2$. However, the central character of $\Sym^4 \bF_5^2$ is trivial; this contradicts to $N\subset \rho_{\fm}(\Gamma_{F(\zeta_p)})$. 
\end{itemize}

\subsection{The case (2)}
Again $n=5$ and now assume that $\rho_{\fm}|_{\Gamma_{F(\zeta_p)}}$ is not irreducible as in (2). (In particular, $5$ divides $[F(\zeta_p):F]$ and $p-1$.) Then, $\rho_{\fm}$ is induced from a character, and if $[F(\zeta_p)\colon F]>5$ we can apply Proposition \ref{induced from character}. 

Suppose $[F(\zeta_p)\colon F]=5$, in which case $\rho_{\fm}$ is induced from a character $\psi$ of $\Gamma_{F(\zeta_p)}$. 
Take a lift $g\in \Gamma_F$ of a generator of $\Gal(F(\zeta_p)/F)$, and we denote $\rho_{\fm}(g)$ by the same symbol $g$. The restriction of $\rho_\fm$ to $\Gamma_{F(\zeta_p)}$ is the direct sum of $\psi^{g^i}$ for $i=0, 1, 2, 3, 4$. 

Let $\rho$ be an irreducible constituent of $H^{i}_{\et}(X_{K, \ov{F}}, \overline{\bF}_p)_{\fm}\otimes \chi_{\fm}$ as a representation of $\Gamma_F$. By \cite{EG}*{4.1.3}, $\rho$ may be regarded as a representation of $\rho_{\fm}(\Gamma_F)$. By Theorem \ref{Key}, Remark \ref{finite extension}, and \cite{EG}*{4.1.4}, the restriction of $\rho$ to $H\coloneqq\rho_{\fm}(\Gamma_{F(\zeta_p)})$ decomposed into characters. 
We first show that $\rho$ itself is not a character of $\Gamma_F$; this implies that $\rho$ is of dimension $5$ and it is induced from a character of $H$. 

\begin{lem}
$\rho$ is not a character of $\Gamma_F$. 
\end{lem}

\begin{proof}
Suppose $\rho$ is a character $\chi$. Then, $\rho_{\fm}\otimes \chi^{-1}$ satisfies, by Theorem \ref{Key} and Remark \ref{finite extension}, the following condition: for any $h\in H$, the multiplicity of $1$ in the multiset of eigenvalues of $(\rho_{\fm}\otimes\chi^{-1})(h)$ is $\ge 2$. 
Up to a permutation, the multiset of eigenvalues of $(\rho_{\fm}\otimes\chi^{-1})(h)$ has two possibilities:
\begin{center}
(a) $1,1,\alpha, \beta,\gamma$, \quad (b) $1, \alpha, 1, \beta, \gamma$, \quad for some $\alpha, \beta, \gamma\neq 0$.
\end{center}
Here, conjugation by $g$ acts by a cyclic permutation $(1\, 2\, 3\, 4\, 5)$. 

In case (a), consideration of $(\rho_{\fm}\otimes\chi^{-1})(h\cdot g^2 hg^{-2})$, $(\rho_{\fm}\otimes\chi^{-1})(h\cdot g^3 hg^{-3})$, gives $\beta=1$ or $\alpha=\gamma=1$. 
If $\beta=1$, consideration of $(\rho_{\fm}\otimes\chi^{-1})(h\cdot g hg^{-1})$, $(\rho_{\fm}\otimes\chi^{-1})(h\cdot g^4 hg^{-4})$ gives $\alpha=1$ or $\gamma=1$. So, the multiplicity of $1$ is $\ge 4$ for every $h\in H$ satisfying (a) up to a permutation. 

Similarly, in case (b), consideration of $(\rho_{\fm}\otimes\chi^{-1})(h\cdot g hg^{-1})$, $(\rho_{\fm}\otimes\chi^{-1})(h\cdot g^4 hg^{-4})$ gives $\alpha=1$ or $\beta=\gamma=1$. 
Since we are in case (i) as well if $\alpha=1$, we deduce that the multiplicity of $1$ is $\ge 4$ for \emph{every} $h\in H$. But, consideration of $(\rho_{\fm}\otimes\chi^{-1})(h\cdot g^i hg^{-i})$ with $i\neq 0$ gives that $(\rho_{\fm}\otimes\chi^{-1})(h)$ is the identity for every $h$. This is impossible.  
\end{proof}

Now, we may assume that $\rho$ is induced from a character of $H$. By \cite{EG}*{4.2.1} and a slight variant of \cite{EG}*{4.2.4 (1)} with the same proof, we deduce that $\det \rho=\det \rho_{\fm}$. 

Take $g' \in H$ such that $\rho (g')$ is not a scalar; such an element exists because  $\rho$ decomposes into distinct characters of $H$. 

\begin{lem}\label{final lemma}
Each eigenvalue of $g'$ has multiplicity $2$ or $3$ in the multiset of eigenvalues of $g'$. 
\end{lem}

\begin{proof}
Let us first observe that $\rho(h)$ has at most two distinct eigenvalues for every $h \in H$ because, if $\rho (h)$ has three distinct eigenvalues, then $h$ must have $6$ eigenvalues by Theorem \ref{Key} and Remark \ref{finite extension}. 

If $\rho(g')$ has eigenvalues $\alpha, \alpha, \alpha, \beta, \beta$ with $\alpha\neq\beta$, then the multiplicities of $\alpha, \beta$ in the set of eigenvalues of $g'$ are both $\geq 2$, and using the equality of determinants, we see that $\rho(g')$ and $g'$ have the same characteristic polynomial. Otherwise, $\rho(g')$ has eigenvalues $\alpha, \alpha, \alpha, \alpha, \beta$ with $\alpha\neq\beta$. (Conjugation by $g$ acts by $(1\, 2\, 3\, 4\, 5)$ as before.)
By applying the observation above to $\rho(g'^2\cdot (g')^{g^i})$ with $i\neq 0$, we see that $\alpha^2=\beta^2$ holds. 
Using the equality of determinants, we deduce that $g'$ has eigenvalues $\alpha, \alpha, \beta, \beta, \beta$. This proves the lemma. 
\end{proof}

Let us complete the proof. 
By a permutation, we may assume that $\psi (g')=\psi^g (g')=\psi^{g^2}(g')$ or $\psi (g')=\psi^g (g')=\psi^{g^3}(g')$. In the former case, $g' \cdot g^2 g'g^{-2}$ has an eigenvalue $\psi (g')^2$ with multiplicity one. In the latter case, $g' \cdot gg'g^{-1}$ has an eigenvalue $\psi(g')^2$ with multiplicity one. These contradict to Lemma \ref{final lemma} with $g'$ replaced by the corresponding elements as $\rho (g' \cdot g^2 g'g^{-2}), \rho (g' \cdot gg'g^{-1})$ are not scalars.

\end{document}